\documentclass[11pt,centertags]{amsart}
\usepackage{amscd}
\usepackage{easybmat}
\usepackage{mathrsfs}
\usepackage{amsfonts}
\usepackage{color}
\usepackage{pifont}
\usepackage{upgreek}
\usepackage{bm}
\usepackage{hyperref}
\usepackage{shorttoc}
\usepackage{amsmath,amstext,amsthm,a4,amssymb,amscd}
\usepackage[mathscr]{eucal}
\usepackage{mathrsfs}
\usepackage{epsf}
\numberwithin{equation}{section}

\def\p{\partial}
\def\o{\overline}
\def\b{\bar}
\def\mb{\mathbb}
\def\mc{\mathcal}
\def\n{\nabla}

\def\t{\triangle}
\def\wt{\widetilde}

\newtheorem{thm}{Theorem}[section]
\newtheorem{lemma}[thm]{Lemma}
\newtheorem{prop}[thm]{Proposition}
\newtheorem{cor}[thm]{Corollary}
\theoremstyle{definition}
\newtheorem{rem}[thm]{Remark}

\newtheorem{ex}[thm]{Example}
\theoremstyle{definition}
\newtheorem{defn}[thm]{Definition}
\newcommand{\comment}[1]{}
\usepackage{fancyhdr}

\begin{document}

\title{Remarks on the geodesic-Einstein metrics of a relative ample line bundle}

\author{Xueyuan Wan, With an appendix by Xu Wang}

\address{Xueyuan Wan: School of Mathematics, KIAS, Heogiro 85, Dongdaemun-gu Seoul, 02455, Republic of Korea}
\email{xwan@kias.re.kr}

\address{Xu Wang: Department of Mathematical Sciences, Norwegian University of Science and Technology, No-7491 Trondheim, Norway}
\email{xu.wang@ntnu.no}

\begin{abstract}
In this paper,  we introduce {the associated geodesic-Einstein flow for a relative ample line bundle $L$ over the total space $\mc{X}$ of a  holomorphic fibration and obtain a few} properties of that flow. {In particular,} we prove that the pair $(\mc{X}, L)$ is nonlinear semistable if the {associated} Donaldson type functional is bounded {from} below and the geodesic-Einstein flow has long-time {existence property}.  We also {define the associated}  $S$-classes and $C$-classes for $(\mc{X}, L)$ and {obtain} two inequalities between them when $L$ admits a geodesic-Einstein metric. Finally, in the appendix of this paper, we prove {that a relative ample line bundle is geodesic-Einstein if and only if an associated infinite rank bundle is Hermitian-Einstein}.
 \end{abstract}
\maketitle
\tableofcontents

\section*{Introduction}

Let $E$ be a holomorphic vector bundle over a compact K\"ahler manifold $(M,\omega)$.  A Hermitian metric $h$ on $E$ is called a Hermitian-Einstein metric if $\Lambda_{\omega}F_h=\lambda \text{Id}$ for some constant $\lambda$, where $F_h=\b{\p}(\p h\cdot h^{-1})\in A^{1,1}(M,\text{End}(E))$ is
the Chern curvature of the Hermitian metric $h$. The famous Donaldson-Uhlenbeck-Yau Theorem reveals the deep relationship between the stability of a holomorphic vector bundle and the existence of Hermitian-Einstein metrics (cf. \cite{Don0, Don1, Don2, Nara, Yau}). On the other hand, in the paper \cite{Feng}, we introduced the notion of the geodesic-Einstein metric on a relative ample line bundle $L$ over $\mc{X}$. For any holomorphic vector bundle $E$, there is a canonical associated projective bundle pair $(P(E), \mc{O}_{P(E)}(1))$. A geodesic-Einstein metric $\phi$ on $\mc{O}_{P(E)}(1)$ is a metric satisfying the geodesic-Einstein equation $tr_{\omega}c(\phi)=\lambda_{\mc{X},\mc{O}_{P(E)}(1)}$ (cf. {Definition \ref{GeoEin}}). In this case, a geodesic-Einstein metric is same as a Finsler-Einstein metric by the natural one-to-one correspondence between Finsler metrics on $E$ and the metrics on the tautological line bundle $\mc{O}_{P(E)}(-1)$ (cf. \cite[page 82]{Ko3}, \cite{Feng}). Combining with \cite[Theorem 0.3]{Feng}, it shows that the existence of geodesic-Einstein metrics is equivalent to the existence of Hermitian-Einstein metrics, which is also equivalent to polystability of the holomorphic vector bundle. %


More generally,  we consider a holomorphic fibration $\pi:\mc{X}\to M$ over a compact complex manifold $M$ with compact fibers, that is, $\pi:\mc{X}\to M$ is a proper surjective holomorphic mapping between complex manifolds $\mc{X}$ and $M$ whose differential has maximal rank everywhere such that every fiber is a compact complex manifold. Let $L$ be a relative ample line bundle over $\mc{X}$ and let $F^+(L)$ denote the space of all metrics $\phi$ such that $\sqrt{-1}\p\b{\p}\phi|_{\pi^{-1}(y)}>0$ for any $y\in M$. A metric $\phi\in F^+(L)$ is geodesic-Einstein if $$tr_{\omega}c(\phi)=\lambda_{\mc{X}, L}.$$ From the above discussions, one sees that the pair $(\mc{X}, L)$ with a geodesic-Einstein metric on $L$ is a natural generalization of Hermitian-Einstein vector bundles and contains many interesting topics studied in complex geometry as its special cases (cf. Section 5 in this paper). There are many interesting problems related to Hermitian-Einstein vector bundles. For examples, the Kobayashi-L\"ubke inequality \cite[Theorem 4.4.7]{Ko3}, the equivalence of the approximate Hermitian-Einstein structure and semistability \cite{Jacob, LZ}, the long times existence of Hermitian-Yang-Mills flow \cite[Proposition 20]{Don1}. Thus, it is natural to ask whether these similar results hold for a general pair $(\mc{X}, L)$ with admitting a geodesic-Einstein metric on $L$. This is also the motivation for the authors to study this topic around the geodesic-Einstein metrics.

Inspired by the Hermitian-Yang-Mills flow \cite{Atiyah, Don1}, for any initial metric $\phi_0\in F^+(L)$, we introduce the following so-called geodesic-Einstein flow:
\begin{align}\label{flow0}
\begin{cases}
&\frac{\p\phi}{\p t}=tr_{\omega}c(\phi)-\lambda\\
&\phi\in F^+(L)\\
&\phi(0)=\phi_0.
\end{cases}
\end{align}
If the pair $(\mc{X}, L)$ is exactly a projective bundle pair $(P(E), \mc{O}_{P(E)}(1))$, then the geodesic-Einstein flow (\ref{flow0}) is reduced to the Hermitian-Yang-Mills flow (see Remark \ref{rem1}). We prove the uniqueness of the solutions of (\ref{flow0}) and establish some estimates in this paper (see Proposition \ref{prop1} and Proposition \ref{prop}).
For any fixed metric $\psi\in F^+(L)$ on $L$ and any $\phi\in F^+(L)$, we also define a {\it Donaldson type functional}
\begin{align}\label{0.2}
    \mc{L}(\phi,\psi)=\int_{M}\left(\frac{\lambda}{m}\mc{E}(\phi,\psi)\wedge\omega-\frac{1}{n+1}\mc{E}_1(\phi,\psi)\right)\frac{\omega^{m-1}}{(m-1)!},
  \end{align}
where $\lambda$ is the constant given by (\ref{lamda}), $m=\dim M$, $n=\dim \mc{X}-\dim M$. Here $\mc{E}(\phi,\psi)$ and $\mc{E}_1(\phi,\psi)$ are defined by (\ref{E0}) and (\ref{E1}). We can prove that the critical points of $\mc{L}(\cdot,\psi)$ {are} exactly the geodesic-Einstein metrics.

Now we recall the definition of the nonlinear semistable pair $(\mc{X},L)$. A fibration $\mc{Y}\to M-S$, where $S$ is a closed subvariety in $M$ of ${\rm codim} S\geq 2$, is called a sub-fibration of the holomorphic fibration ${\mc X}\to M$ if for any $p\in M-S$, the fiber ${\mc Y}_p$ is a closed complex submanifold of the fiber ${\mc X}_p$. Let $\mc{F}$ be the set of sub-fibrations of the holomorphic fibration ${\mc X}\to M$.
For any $\mc{Y}\in \mathscr{F}$, we set
\begin{align*}
\lambda_{\mc{Y},L}=\frac{2\pi m}{{\dim\mc{Y}/M}+1}\frac{([\omega]^{m-1}c_1(L)^{{\dim\mc{Y}/M}+1})[\mc{Y}]}{([\omega]^mc_1(L)^{\dim\mc{Y}/M})[\mc{Y}]}.
\end{align*}
Note that $\lambda_{\mc{Y},L}$ is well-defined and independent of the metrics on $L$ by Stokes' theorem and ${\rm codim} S\geq 2$. Similar to the semistability of a holomorphic vector bundle,  a pair $(\mc{X}, L)$ is called {\it nonlinear semistable}  if $\lambda_{\mc{Y},L}\geq \lambda_{\mc{X},L}$  for any sub-fibration $\mc{Y}\in \mathscr{F}$ (see Definition \ref{stability}).

 A pair  $(\mc{X},L)$ admits an {\it approximate geodesic-Einstein structure} if for any given $\epsilon>0$, there exists an metric $\phi_{\epsilon}$ on $L$ such that
  $$\max_{ M}\left|\int_{\mc{X}/M}(tr_{\omega}c(\phi_{\epsilon})-\lambda)^2(\sqrt{-1}\p\b{\p}\phi_{\epsilon})^n\right|^{\frac{1}{2}}< \epsilon.$$
\begin{thm}\label{thm0.1}
  Suppose that the geodesic-Einstein flow (\ref{flow0}) has a smooth solution for $0\leq t<+\infty$, then we have implications $(1)\Rightarrow (2)\Rightarrow (3)$ for the following statements:
  \begin{itemize}
  \item[(1)] the  functional $\mc{L}$ is bounded from below;
  \item[(2)]  $L$ admits an approximate geodesic-Einstein structure;
  \item[(3)] the pair $(\mc{X},L)$ is nonlinear semistable.
  \end{itemize}
\end{thm}

 We also introduce the $S$-class and $C$-class associated to the pair  $(\mc{X},L)$, which are the generalizations of the Segre class and Chern class. For a general pair $(\mc{X},L)$, we define the total $S$-class and the $S$-classes of $L$ by
\begin{align*}
 S(L)=\sum_{i=0}^m S_i(L),\quad S_i(L)=\pi_{*}((c_1(L))^{n+i})\in H^{2i}(M,\mb{Z}), \quad 0\leq i\leq m,
\end{align*}
and the total $C$-class and the $C$-classes are defined by
\begin{align*}
C(L)=\frac{1}{S(L)},\quad 	C(L)=\sum_{i=0}^m C_i(L),\quad C_i(L)\in H^{2i}(M,\mb{Q}).
\end{align*}

\begin{thm}\label{thm0.3}
\begin{enumerate}
  \item  If $\phi$ is a geodesic-Einstein metric on $L$, i.e., $tr_{\omega}c(\phi)=\lambda_{\mc{X},L}$, then
  \begin{align*}
    S_2(L,\phi)\wedge \omega^{m-2}\leq \frac{(n+1)(n+2)}{8\pi^2 m^2}\lambda^2_{\mc{X},L}S_0(L)\omega^m,
  \end{align*}
  the equality holds if and only if $c(\phi)=\frac{\lambda_{\mc{X},L}}{m}\omega$. In particular,
  \begin{align*}
  \int_{M} S_2(L)\wedge [\omega]^{m-2}\leq 	\frac{(n+1)(n+2)}{8\pi^2 m^2}\lambda^2_{\mc{X},L}S_0(L)\int_{M}[\omega]^m.
  \end{align*}
\item  If $\phi$ is a geodesic-Einstein metric on $L$, then
  \begin{align*}
  (nC_1(L,\phi)^2-2(n+1)C_0(L)C_2(L,\phi))\wedge \omega^{m-2}\leq 0,
  \end{align*}
  the equality holds if and only if $c(\phi)=\frac{2\pi}{(n+1)S_0(L)}S_1(L,\phi)$. In particular,
  \begin{align*}
  \int_M (nC_1(L)^2-2(n+1)C_0(L)C_2(L))\wedge [\omega]^{m-2}\leq 0.	
  \end{align*}
\end{enumerate}
\end{thm}

We also consider the Hermitian-Einstein metrics on a quasi-vector bundle. Let $\pi: \mc{X}\to B$ be a proper holomorphic submersion from a complex manifold $\mc{X}$ to another complex manifold $B$ (need not compact), each fiber $X_t:=\pi^{-1}(t)$ is an $n$-dimensional compact complex manifold;
 $E$  a holomorphic vector bundle over $\mc{X}$, $E_t:=E|_{X_t}$;
 $\omega$  a d-closed $(1,1)$-form on $\mc{X}$ and is positive on each fiber, $\omega^t:=\omega|_{X_t}$;
 $h_E$ is a smooth Hermitian metric on $E$, $h_{E_t}:=h_E|_{E_t}$. For each $t\in B$, let us denote by $\mc{A}^{p,q}(E_t)$ the space of smooth $E_t$-valued $(p,q)$-forms on $X_t$. Denote
$$\mc{A}^{p,q}:=\{\mc{A}^{p,q}(E_t)\}_{t\in B},$$
and let $(\mc{A},\Gamma)$ denote the corresponding quasi-vector bundle (see Section \ref{Sec HE}). There is a natural Chern connection $D^{\mc{A}}$ on each $\mc{A}^{p,q}$ with respect the standard $L^2$-metric. The $L^2$-metric  on  $\mc{A}^{p,q}$ is called Hermitian-Einstein with respect to a Hermitian metric $\omega_B=\sqrt{-1}g_{i\b{j}}dt^i\wedge d\b{t}^j$ if
\begin{align}
\Lambda_{\omega_B}(D^{\mc{A}})^2=\lambda \text{Id}	
\end{align}
for some constant $\lambda$ (see Definition \ref{defn of HE}). Now we consider the case $p=q=0$ and $E$ is the trivial bundle. Let $L$ be a relative ample line bundle over $\mc{X}$, i.e. there exists a metric $\phi$ on $L$ such that its curvature $\omega:=\sqrt{-1}\p\b{\p}\phi$ is positive along each fiber. The $L^2$-metric  on $\mc{A}^{0,0}$ is given by
\begin{align}\label{L2-metric1}
	(u,v)=\int_{X_t} u\b{v}\frac{\omega^n}{n!}.
\end{align}
We call a metric $\phi$ on $L$ is weak geodesic-Einstein with respect to $\omega_B$ if $tr_{\omega_B}c(\phi)=\pi^*f(z)$ for some function $f(z)$ on $B$. Then
\begin{prop}\label{prop0.8}
	$\phi$ is a weak geodesic-Einstein metric on $L$ if and only if the $L^2$-metric (\ref{L2-metric1}) is a Hermitian-Einstein metric on $\mc{A}^{0,0}$. In particular, if $B$ is compact, then up to a smooth function on $B$, $\phi$ is a  geodesic-Einstein metric on $L$ if and only if the  $L^2$-metric (\ref{L2-metric1}) is a Hermitian-Einstein metric on $\mc{A}^{0,0}$.
\end{prop}

This article is organized as follows. In Section \ref{sec1}, we will recall some basic definitions and facts on geodesic-Einstein metrics of a relative ample line bundle over a holomorphic fibration. In Section \ref{sec2}, we will define the geodesic-Einstein flow (\ref{flow0}) and the Donaldson type functional (\ref{0.2}),  and Theorem \ref{thm0.1} will be proved in this section. In Section \ref{sec3}, we will discuss the case of $tr_{
\omega}c(\phi)\geq 0$. In Section \ref{sec4}, we will give the definitions of $S$-class and $C$-class and prove Theorem \ref{thm0.3}.  In Section \ref{sec example}, we will discuss some examples around geodesic-Einstein metrics. In the last section, we will give the equivalence between the geodesic-Einstein metric on a relative ample line bundle and the Hermitian-Einstein metric on a quasi-vector bundle, and prove Proposition \ref{prop0.8}.


\section{Preliminaries}\label{sec1}

In this section, we will recall some basic definitions and facts on geodesic-Einstein metrics of a relative ample line bundle over holomorphic fibration. For more details one may refer to \cite{Feng, Wan1}.

Let $\pi:\mc{X}\to M$ be a holomorphic fibration over a compact complex manifold $M$ with compact fibers. We denote by
$(z;v)=(z^1,\cdots, z^m; v^1,\cdots, v^n)$ a local admissible holomorphic coordinate system of $\mc{X}$ with $\pi(z;v)=z$, where $m=\dim_{\mb{C}}M$, $n=\dim_{\mb{C}}\mc{X}-\dim_{\mb{C}}M$.

For any smooth function $\phi$ on $\mc{X}$, we denote
$$\phi_{\alpha}:=\frac{\p \phi}{\p z^{\alpha}},\quad \phi_{\b{\beta}}:=\frac{\p \phi}{\p \b{z}^{\beta}},
\quad \phi_{i}:=\frac{\p \phi}{\p v^i},\quad \phi_{\b{j}}:=\frac{\p \phi}{\p \b{v}^j},$$
where $1\leq i,j\leq n, 1\leq \alpha,\beta\leq m$.

For any holomorphic line bundle $L$ over $\mc{X}$, we denote by
$F^+(L)$  the space of smooth metrics $\phi$ on $L$ with
\begin{align*}
(\sqrt{-1}\p\b{\p}\phi)|_{\mc{X}_y}>0	
\end{align*}
for any point $y\in M$. Now we assume that $L$ is a relative ample line bundle, i.e. $F^+(L)\neq \emptyset$. For any $\phi\in F^+(L)$, set
\begin{align}\label{horizontal}
  \frac{\delta}{\delta z^{\alpha}}:=\frac{\p}{\p z^{\alpha}}-N^k_{\alpha}\frac{\p}{\p v^k},
\end{align}
where $N^k_{\alpha}=\phi_{\alpha\b{j}}\phi^{\b{j}k}$.
By a routine computation, one can show that $\{\frac{\delta}{\delta z^{\alpha}}\}_{1\leq \alpha\leq m}$ spans a well-defined horizontal subbundle of $T\mc{X}$ (see \cite[Section 1]{Feng}).

 Let $\{dz^{\alpha};\delta v^k\}$
denote the dual frame of $\left\{\frac{\delta}{\delta z^{\alpha}}; \frac{\p}{\p v^i}\right\}$. Then
$$\delta v^k=dv^k+\phi^{k\b{l}}\phi_{\b{l}\alpha}dz^{\alpha}.$$
Moreover, the differential operators
\begin{align}\label{HV}
\p^V=\frac{\p}{\p v^i}\otimes \delta v^i,\quad \p^H=\frac{\delta}{\delta z^{\alpha}}\otimes dz^{\alpha}.
\end{align}
are well-defined.

For any $\phi\in F^+(L)$, the geodesic curvature $c(\phi)$ is defined by
\begin{align*}
  c(\phi)=\left(\phi_{\alpha\b{\beta}}-\phi_{\alpha\b{j}}\phi^{i\b{j}}\phi_{i\b{\beta}}\right)\sqrt{-1} dz^{\alpha}\wedge d\b{z}^{\beta},
\end{align*}
which is a horizontal real $(1,1)$-form on $\mc X$.
Then we have the following decomposition.
\begin{lemma}[{\cite[Lemma 1.1]{Feng}}]\label{lemma0} The following decomposition holds,
  \begin{align*}
    \sqrt{-1}\p\b{\p}\phi=c(\phi)+\sqrt{-1}\phi_{i\b{j}}\delta v^i\wedge \delta \b{v}^j.
  \end{align*}
\end{lemma}

From \cite[Definition 1.2]{Feng}, the geodesic-Einstein metric is defined as follows.
\begin{defn}\label{GeoEin} Let $\omega=\sqrt{-1}g_{\alpha\b{\beta}}dz^{\alpha}\wedge d\b{z}^{\beta}$ be a (fixed) K\"{a}hler metric on $M$.
  A metric $\phi\in F^+(L)$ is called a geodesic-Einstein metric on $L$ with respect to $\omega$ if it satisfies that
  \begin{align}\label{FE}
    tr_{\omega}c(\phi):=g^{\alpha\b{\beta}}\left(\phi_{\alpha\b{\beta}}-\phi_{\alpha\b{j}}\phi^{i\b{j}}\phi_{i\b{\beta}}\right)=\lambda,
  \end{align}
  where $\lambda$ is a constant. By \cite[Proposition 1.3]{Feng}, if $M$ is compact, $\lambda$ is a topological quantity, which is given by
  \begin{align}\label{lamda}
     \lambda=\frac{2\pi m}{n+1}\frac{([\omega]^{m-1}c_1(L)^{n+1})[\mc{X}]}{([\omega]^mc_1(L)^n)[\mc{X}]}.
   \end{align}
\end{defn}


\section{Geodesic-Einstein flow and nonlinear semistable}\label{sec2}

For any fixed metric $\psi\in F^+(L)$ on $L$ and any $\phi\in F^+(L)$, one can define the following two functionals $\mc{E}, \mc{E}_1$:
 \begin{align}\label{E0}
   \mc{E}(\phi,\psi)=\frac{1}{n+1}\int_{\mc{X}/M}(\phi-\psi)\sum_{k=0}^n(\sqrt{-1}\p\b{\p}\phi)^k\wedge(\sqrt{-1}\p\b{\p}\psi)^{n-k}
 \end{align}
and
\begin{align}\label{E1}
  \mc{E}_1(\phi,\psi) =\frac{1}{n+2}\int_{\mc{X}/M}(\phi-\psi)\sum_{k=0}^{n+1}(\sqrt{-1}\p\b{\p}\phi)^k\wedge(\sqrt{-1}\p\b{\p}\psi)^{n+1-k}.
\end{align}
Note that $\mc{E}(\phi,\psi)$ is a smooth function, while $\mc{E}_1(\phi,\psi)$ is a smooth real $(1,1)$-form on $M$.

From \cite[(1.14)]{Feng}, the Donaldson type functional $\mc{L}$ on $F^+(L)$ is defined by
\begin{align}\label{D}
    \mc{L}(\phi,\psi)=\int_{M}\left(\frac{\lambda}{m}\mc{E}(\phi,\psi)\wedge\omega-\frac{1}{n+1}\mc{E}_1(\phi,\psi)\right)\frac{\omega^{m-1}}{(m-1)!},
  \end{align}
where $\lambda$ is the constant given by (\ref{lamda}).  Let $\phi_t$ be a smooth family of metrics depends on $t$, then the first variation of Donaldson type function is given by
\begin{align}\label{var}
\frac{d}{d t}\mc{L}(\phi_t,\psi)=-\int_{\mc{X}}\dot{\phi}_t(tr_{\omega}c(\phi_t)-\lambda)(\sqrt{-1}\p\b{\p}\phi_t)^n\wedge\frac{\omega^m}{m!},
\end{align}
(see \cite[(1.15)]{Feng}). So $\phi\in F^+(L)$ is a geodesic-Einstein metric if and only if
it is a critical point of $\mc{L}(\cdot,\psi)$ on $F^+(L)$ (see \cite[Proposition 1.4]{Feng}).

In order to make $\frac{d}{d t}\mc{L}(\phi_t,\psi)\leq 0$, it is natural to consider the following {\it geodesic-Einstein flow}
\begin{align}\label{flow}
\begin{cases}
&\frac{\p\phi}{\p t}=tr_{\omega}c(\phi)-\lambda\\
&\phi\in F^+(L)\\
&\phi(0)=\phi_0
\end{cases}
\end{align}
for an initial metric $\phi_0\in F^+(L)$. For the convenience, we also denote $\dot{\phi}_t:=\frac{\p\phi}{\p t}$
and $\ddot{\phi}_t=\frac{\p^2\phi}{\p t^2}$.

\subsection{Some properties of the geodesic-Einstein flow}

In this subsection, we will study some properties of the geodesic-Einstein flow (\ref{flow}) by  using the method of studying K\"ahler-Ricci flow (see e.g. \cite{Bou, Tos}).

For any smooth function $f\in C^{\infty}(\mc{X})$, we denote the horizontal and vertical Laplacian by
\begin{align}\label{lapa}
  \t_{\omega}f:=g^{\alpha\b{\beta}}(\p\b{\p}f)(\frac{\delta}{\delta z^{\alpha}},\frac{\delta}{\delta \b{z}^{\beta}}),\quad \t_{\phi}f:=\phi^{k\b{l}}\frac{\p^2 f}{\p v^k\p\b{v}^l},
\end{align}
respectively.
The following proposition is a maximum (minimum) principle for degenerate parabolic elliptic equations. Its proof is the same as \cite[Proposition 3.1.7]{Bou}.
\begin{prop}[{\cite[Proposition 3.1.7]{Bou}}]\label{max}
  Fix $T$ with $0<T\leq \infty$. Suppose that $f=f(x,t)$ is a smooth function on $\mc{X}\times [0,T)$ satisfying
  \begin{align}
    \left(\frac{\p}{\p t}-\t_{\omega}\right)f\leq 0 \,\,(resp. \geq 0).
  \end{align}
  Then $\sup_{(x,t)\in \mc{X}\times [0,T)}f(x,t)\leq \sup_{x\in \mc{X}}f(x,0)$ (resp. $\inf_{(x,t)\in \mc{X}\times [0,T)}f(x,t)\geq \inf_{x\in \mc{X}}f(x,0)$ ).
\end{prop}

By above proposition, we can prove the uniqueness of the solutions of the flow (\ref{flow}).
\begin{prop}\label{prop1}
  If $\phi(t)$ and $\psi(t)$ are the two solutions of the flow (\ref{flow}), then $\phi(t)=\psi(t)$.
\end{prop}
\begin{proof}
  We assume that $T_{\max}$ is the maximum existence time and let $\tau<T_{\max}$. Let $\phi(t)$ and $\psi(t)$ be the two solutions with the same initial metric $\phi_0$. Then
  $$(\phi-\psi)'_t=tr_{\omega}c(\phi)-tr_{\omega}c(\psi),\quad (\phi-\psi)(0)=0.$$
  We assume that on $\mc{X}\times [0,\tau]$, the maximum of $\phi-\psi-\epsilon t$ is taken at $(x,t)$ in $\mc{X}\times [0,\tau]$, $t>0$. Then
  \begin{align}\label{1.2}
    0\leq (\phi-\psi-\epsilon t)'_t=tr_{\omega}c(\phi)-tr_{\omega}c(\psi)-\epsilon
  \end{align}
  and $-\sqrt{-1}\p\b{\p}(\phi-\psi-\epsilon t)=-\sqrt{-1}\p\b{\p}(\phi-\psi)$ is semi-positive  $(1,1)$-form at the point $(x,t)$. Denote by $\frac{\delta_{\phi}}{\delta z^{\alpha}}$ (resp. $\frac{\delta_{\psi}}{\delta z^{\alpha}}$) the horizontal lifts with respect to $\phi$ (resp. $\psi$), which are given by (\ref{horizontal}). Then at the point $(x,t)$, one has
  \begin{align}\label{1.1}
  \begin{split}
  	tr_{\omega}c(\psi)&=g^{\alpha\b{\beta}}(\p\b{\p}\psi)(\frac{\delta_{\psi}}{\delta z^{\alpha}},\frac{\delta_{\psi}}{\delta \b{z}^{\beta}})\\
  	&\geq g^{\alpha\b{\beta}}(\p\b{\p}\phi)(\frac{\delta_{\psi}}{\delta z^{\alpha}},\frac{\delta_{\psi}}{\delta \b{z}^{\beta}})\\
  	&=g^{\alpha\b{\beta}}(\p\b{\p}\phi)\left((\frac{\delta_{\psi}}{\delta z^{\alpha}}-\frac{\delta_{\phi}}{\delta z^{\alpha}})+\frac{\delta_{\phi}}{\delta z^{\alpha}},(\frac{\delta_{\psi}}{\delta \b{z}^{\beta}}-\frac{\delta_{\phi}}{\delta \b{z}^{\beta}})+\frac{\delta_{\phi}}{\delta \b{z}^{\beta}}\right)\\
  	&=tr_{\omega}c(\phi)+g^{\alpha\b{\beta}}(\p\b{\p}\phi)\left(\frac{\delta_{\psi}}{\delta z^{\alpha}}-\frac{\delta_{\phi}}{\delta z^{\alpha}},\frac{\delta_{\psi}}{\delta \b{z}^{\beta}}-\frac{\delta_{\phi}}{\delta \b{z}^{\beta}}\right)\geq tr_{\omega}c(\phi),
  \end{split}	
  \end{align}
where the fourth equality holds by Lemma \ref{lemma0} and noting that $\frac{\delta_{\psi}}{\delta \b{z}^{\beta}}-\frac{\delta_{\phi}}{\delta \b{z}^{\beta}}$ is a vertical vector, the last inequality follows from $\phi\in F^+(L)$.

  Substituting (\ref{1.1}) into (\ref{1.2}) we get a contradiction. So the maximum of $\phi-\psi-\epsilon t$ is taken at $t=0$, i.e.
   $$\max_{x\in\mc{X}\times[0,\tau]}(\phi-\psi-\epsilon t)=\max_{x\in\mc{X}}(\phi(0)-\psi(0))=0.$$
   Thus
   $$\phi-\psi\leq \epsilon t\leq \epsilon\tau.$$
   It follows that
   $$\phi\leq \psi$$
   for any $[0,\tau], \tau<T_{\max}$. Thus $\phi\leq \psi$ on $[0,T_{\max})$. Similarly, we have
   $$\psi\leq \phi.$$
   Therefore, $\psi(t)=\phi(t)$.
\end{proof}

Also by Proposition \ref{max}, we obtain the following estimates.
\begin{prop}\label{prop} Along the flow (\ref{flow}), one has
\begin{align}
|tr_{\omega}c(\phi)|<C \quad\text{and}\quad  |\phi(t)-\phi_0|\leq Ct 	
\end{align}
	for some constant $C>0$.
\end{prop}
\begin{proof}
By Lemma \ref{lemma0} and (\ref{flow}), one has
\begin{align}\label{geolam}
	 (\dot{\phi}_t+\lambda)\omega^m\wedge(\sqrt{-1}\p\b{\p}\phi)^n=tr_{\omega}c(\phi)\omega^m\wedge(\sqrt{-1}\p\b{\p}\phi)^n=\frac{m}{n+1}\omega^{m-1}\wedge (\sqrt{-1}\p\b{\p}\phi)^{n+1}.
\end{align}

Taking derivative on $t$ to the both sides of above equation, then
\begin{align}\label{left}
  \begin{split}
  &\quad \frac{d}{dt}(\dot{\phi}_t+\lambda)\omega^m\wedge(\sqrt{-1}\p\b{\p}\phi)^n \\
   &=\ddot{\phi}_t\omega^m\wedge (\sqrt{-1}\p\b{\p}\phi)^n+(\dot{\phi}_t+\lambda)\omega^m\wedge (\sqrt{-1}\p\b{\p}\phi)^{n-1}\wedge (\sqrt{-1}\p\b{\p}\dot{\phi}_t)\\
    &=(\ddot{\phi}_t+tr_{\omega}c(\phi)\t_{\phi}\dot{\phi}_t)\omega^m\wedge (\sqrt{-1}\p\b{\p}\phi)^n
  \end{split}
\end{align}
and
\begin{align}\label{right}
  \begin{split}
  &\quad \frac{m}{n+1}\frac{d}{d t}\omega^{m-1}\wedge (\sqrt{-1}\p\b{\p}\phi)^{n+1}\\
    &= m\omega^{m-1}\wedge (\sqrt{-1}\p\b{\p}\phi)^n\wedge (\sqrt{-1}\p\b{\p}\dot{\phi}_t)\\
    &=m\omega^{m-1}\wedge (c(\phi)+\sqrt{-1}\phi_{i\b{j}}\delta v^i\wedge \delta\b{v}^j)^n\wedge (\sqrt{-1}\p\b{\p}\dot{\phi}_t)\\
    &=(\t_{\omega}\dot{\phi}_t+tr_{\omega}c(\phi)\t_{\phi}\dot{\phi}_t)\omega^m\wedge (\sqrt{-1}\p\b{\p}\phi)^n.
  \end{split}
\end{align}
From (\ref{left}) and (\ref{right}), one has
\begin{align}\label{sec5.1}
\left(\frac{\p}{\p t}-\t_{\omega}\right)\dot{\phi}_t=0.
\end{align}
From Proposition \ref{max}, one has
 \begin{align}\label{1.3}
 	|\dot{\phi}_t|<C_1
 \end{align}
 for some constant $C_1>0$. It follows that
\begin{align}
  |tr_{\omega}c(\phi)|=|\dot{\phi}_t+\lambda|\leq |\dot{\phi}_t|+|\lambda|\leq C_1+|\lambda|=:C.
\end{align}
By (\ref{1.3}), one has
\begin{align}
|\phi(t)-\phi_0|=\left|\int^t_0 \dot{\phi}_t dt\right|\leq \int^t_0|\dot{\phi}_t|dt\leq C_1 t\leq Ct.
\end{align}
\end{proof}

\begin{rem}\label{rem1}
If the total space $\mc{X}$ is a projective bundle and $L$ is a hyperplane line bundle over $\mc{X}$, then this flow (\ref{flow}) has been studied in \cite{Wan}. More precisely, let $E\to M$ be a holomorphic vector bundle over $M$, $\mc{X}:=P(E)$ denotes the projective bundle of $E$, and $L=\mc{O}_{P(E)}(1)$ the hyperplane line bundle.
Let $h$ be a Hermitian metric on $E$, then it induces a  metric on the line bundle $L=\mc{O}_{P(E)}(1)$ by
$$\phi^i=-\log \frac{|v^i|^2}{G}$$
where $G=\sum_{i,j=1}^r h_{i\b{j}}v^i\b{v}^j$. So
$G_{i\b{j}}:=\frac{\p^2 G}{\p v^i\p\b{v}^j}=h_{i\b{j}}$
is independent of the fibers. Suppose that the initial metric $\phi_0$ is induced from a Hermitian metric $h_0$ on $E$,
then the flow (\ref{flow}) is equivalent to
       \begin{align}\label{1.6}
       \begin{split}
           0&=\frac{\p \phi}{\p t}-tr_{\omega}c(\phi)+\lambda\\
           &=\frac{1}{G}\frac{\p G}{\p t}-g^{\alpha\b{\beta}}(\phi_{\alpha\b{\beta}}-\phi^{k\b{l}}\phi_{k\b{\beta}}\phi_{\alpha\b{l}})+\lambda\\
           &=\frac{v^i\b{v}^j}{G}\left(\frac{\p G_{i\b{j}}}{\p t}-g^{\alpha\b{\beta}}((G_{i\b{j}})_{\alpha\b{\beta}}-G^{k\b{l}}G_{k\b{j}\b{\beta}}G_{\alpha i\b{l}})+\lambda G_{i\b{j}}\right).
           \end{split}
         \end{align}
      By the argument of \cite[Remark 2.1]{Wan} or by the uniqueness of solution (Proposition \ref{prop1} and \cite[Corollary 1.4]{Don1}), (\ref{1.6}) is reduced to
         \begin{align}\label{Yang-mills}
           G^{-1}\cdot \frac{\p G}{\p t}:=G^{i\b{l}}\frac{\p G_{j\b{l}}}{\p t}=-g^{\alpha\b{\beta}}F_{j\alpha\b{\beta}}^i+\lambda \delta^i_j=-\Lambda F+\lambda I,
         \end{align}
        which is exactly the Hermitian-Yang-Mills flow (cf. \cite{Atiyah, Don1}), where the curvature operator $F$ is defined by
         $$F:=\sqrt{-1}\b{\p}(\p G\cdot G^{-1})\in A^{1,1}(M, \text{End}(E)).$$
So the flow (\ref{flow}) is indeed a natural generalization of Hermitian-Yang-Mills flow.
By \cite[Proposition 20]{Don1}, the Hermitian-Yang-Mills flow (\ref{Yang-mills}) has a unique smooth solution for $0\leq t<+\infty$. Immediately, one natural and interesting problem is that whether the geodesic-Einstein flow (\ref{flow}) has a smooth solution for $0\leq t<+\infty$.
\end{rem}

\subsection{Nonlinear semistable}

In this subsection, we will assume that the geodesic-Einstein flow (\ref{flow}) has a smooth solution for $0\leq t<+\infty$ and consider nonlinear semistability of a pair $(\mc{X}, L)$ (see Definition \ref{defn1}).

Firstly, we recall the definition of nonlinear semistable.
A fibration $\mc{Y}\to M-S$, with $S$ a closed subvariety in $M$ of ${\rm codim} S\geq 2$, is called a sub-fibration of the holomorphic fibration ${\mc X}\to M$ if for any $p\in M-S$, the fiber ${\mc Y}_p$ is a closed complex  submanifold of the fiber ${\mc X}_p$. Let $\mc{F}$ be the set of sub-fibrations of the holomorphic fibration ${\mc X}\to M$.
For any $\mc{Y}\in \mathscr{F}$, we set
\begin{align}\label{l}
\lambda_{\mc{Y},L}=\frac{2\pi m}{{\dim\mc{Y}/M}+1}\frac{([\omega]^{m-1}c_1(L)^{{\dim\mc{Y}/M}+1})[\mc{Y}]}{([\omega]^mc_1(L)^{\dim\mc{Y}/M})[\mc{Y}]}.
\end{align}
Note that $\lambda_{\mc{Y},L}$ is well-defined and independent of the metrics on $L$ by the Stoke's theorem and ${\rm codim} S\geq 2$. Similar to the semistability of a holomorphic vector bundle, the nonlinear semistable of a pair $(\mc{X},L)$ is given by the following.
\begin{defn}[{\cite[Definition 2.1]{Feng}}]\label{stability}
    A pair $(\mc{X}, L)$ is called nonlinear
  semistable  if $\lambda_{\mc{Y},L}\geq \lambda_{\mc{X},L}$  for any sub-fibration $\mc{Y}\in \mathscr{F}$.
\end{defn}
Now we assume that the flow (\ref{flow}) has a smooth solution for $0\leq t<+\infty$, we obtain
\begin{prop}\label{prop2}
 Suppose that the geodesic-Einstein flow (\ref{flow}) has a smooth solution $\phi_t=\phi(t)$ for $0\leq t<+\infty$, then
  \begin{enumerate}
    \item The Donaldson type functional is monotone decreasing function of $t$; in fact,
        \begin{align}\label{fir var}
          \frac{d}{dt}\mc{L}(\phi(t),\phi_0)=-\int_{\mc{X}}(tr_{\omega}c(\phi_t)-\lambda)^2(\sqrt{-1}\p\b{\p}\phi_t)^n\wedge\frac{\omega^m}{m!}\leq 0.
        \end{align}
    \item $\max_{\mc{X}}(tr_{\omega}c(\phi)-\lambda)^2$ is a monotone decreasing function of $t$;
    \item If $\mc{L}(\phi(t),\phi_0)$ is bounded below, i.e., $\mc{L}(\phi(t),\phi_0)\geq A>-\infty$ for $0\leq t<+\infty$, then
        $$\max_{M}\int_{\mc{X}/M}(tr_{\omega}c(\phi_t)-\lambda)^2(\sqrt{-1}\p\b{\p}\phi_t)^n\to 0$$
        as $t\to +\infty$.
  \end{enumerate}
\end{prop}

\begin{proof}
  \begin{enumerate}
     \item Substituting  (\ref{flow}) into the first variation (\ref{var}) of the Donaldson type functional, one has
     \begin{align*}
     \frac{d}{d t}\mc{L}(\phi(t),\phi_0)&=-\int_{\mc{X}}\dot{\phi}_t(tr_{\omega}c(\phi_t)-\lambda)(\sqrt{-1}\p\b{\p}\phi_t)^n\wedge\frac{\omega^m}{m!}\\
     &=-\int_{\mc{X}}(tr_{\omega}c(\phi_t)-\lambda)^2(\sqrt{-1}\p\b{\p}\phi_t)^n\wedge\frac{\omega^m}{m!}\leq 0.
     \end{align*}

     \item By a direct calculation, one has
     \begin{align*}
       &\frac{1}{2}\left(\frac{\p}{\p t}-\t_{\omega}\right)(tr_{\omega}c(\phi)-\lambda)^2)=\frac{1}{2}\left(\frac{\p}{\p t}-\t_{\omega}\right)\dot{\phi}^2_t\\
       &=\dot{\phi}_t\ddot{\phi}_{tt}-\frac{1}{2}g^{\alpha\b{\beta}}(\p\b{\p}\dot{\phi}^2_t)(\frac{\delta}{\delta z^{\alpha}},\frac{\delta}{\delta\b{z}^{\beta}})\\
       &=\dot{\phi}_t\ddot{\phi}_{tt}-\dot{\phi}_t\t_{\omega}\dot{\phi}_t-|\p^H\dot{\phi}_t|^2\\
       &=\dot{\phi}_t\left(\frac{\p}{\p t}-\t_{\omega}\right)\dot{\phi}_t-|\p^H\dot{\phi}_t|^2\\
       &=-|\p^H\dot{\phi}_t|^2\leq 0,
     \end{align*}
     where the last equality holds by (\ref{sec5.1}).
     By Proposition \ref{max}, we complete the proof.

 \item Integrating (\ref{fir var}) from $0$ to $s$, we obtain
$$\mc{L}(\phi(s),\phi_0)-\mc{L}(\phi(0),\phi_0)=-\int_0^s
\int_{\mc{X}}(tr_{\omega}c(\phi_t)-\lambda)^2(\sqrt{-1}\p\b{\p}\phi_t)^n\wedge\frac{\omega^m}{m!} ds$$
Since the $\mc{L}$ is bounded below by a constant independent of $s$, we have
$$\int_0^{\infty}\int_{\mc{X}}(tr_{\omega}c(\phi_t)-\lambda)^2(\sqrt{-1}\p\b{\p}\phi_t)^n\wedge\frac{\omega^m}{m!} ds<\infty.$$
In particular,
\begin{align}
  \int_{\mc{X}}(tr_{\omega}c(\phi_t)-\lambda)^2(\sqrt{-1}\p\b{\p}\phi_t)^n\wedge\frac{\omega^m}{m!}\to 0
  \quad t\to\infty.
\end{align}
Let $H(z,w,t)$ be the heat kernel for $\p_t-\t_{\omega}$ when acting on $C^{\infty}(M)$. Set
$$F(z,t)=\left(\int_{\mc{X}/M}(tr_{\omega}c(\phi_t)-\lambda)^2(\sqrt{-1}\p\b{\p}\phi_t)^n\right)(z),\quad (z,t)\in M\times [0,\infty).$$
Fix $t_0\in [0,\infty)$ and set
 $$u(z,t)=\int_M H(z,w,t-t_0)F(w,t_0)dw, \quad dw=\frac{\omega^m}{m!}.$$
Then $u(z,t)$ is of class $C^{\infty}$ on $M\times (t_0,\infty)$ and extends to a continuous function on $M\times [t_0,\infty)$. It satisfies
\begin{equation*}
  \begin{cases}
    &(\p_t-\t_{\omega})u(z,t)=0\quad (x,t)\in M\times (t_0,\infty),\\
    &u(z,t_0)=F(z,t_0) \quad z\in M.
  \end{cases}
\end{equation*}
And we have
$$\frac{\p F(z,t)}{\p t}=\pi_*\left((\frac{\p}{\p t}(tr_{\omega}c(\phi_t)-\lambda)^2+\t_{\phi}\dot{\phi}(tr_{\omega}c(\phi_t)-\lambda)^2)(\sqrt{-1}\p\b{\p}\phi)^n\right)$$
and
$$\t_{\omega}F(z,t)=\pi_*((\t_{\omega}(tr_{\omega}c(\phi_t)-\lambda)^2+\t_{\phi}(tr_{\omega}c(\phi_t)-\lambda)^2tr_{\omega}c(\phi))(\sqrt{-1}\p\b{\p}\phi)^n).$$
  By (\ref{sec5.1}) and Stoke's Theorem, one has
\begin{align}
  \begin{split}
    &\left (\frac{\p}{\p t}-\t_{\omega}\right)F(z,t)\\
    &=\pi_*((-2|\p^H\dot{\phi}_t|^2+(\t_{\phi}\dot{\phi}_t)\dot{\phi}^2_t-(\t_{\phi}\dot{\phi}^2_t)(\dot{\phi}+\lambda))(\sqrt{-1}\p\b{\p}\phi)^n)\\
    &=\pi_*((-2|\p^H\dot{\phi}_t|^2-\frac{1}{3}\t_{\phi}(\dot{\phi}_t^3+3\lambda\dot{\phi}_t))(\sqrt{-1}\p\b{\p}\phi)^n)\\
    &=\pi_*(-2|\p^H\dot{\phi}_t|^2(\sqrt{-1}\p\b{\p}\phi)^n)\leq 0.
  \end{split}
\end{align}
   Thus
   $$(\p_t-\t_{\omega})(F(z,t)-u(z,t))\leq 0,\quad (z,t)\in M\times (t_0,\infty).$$
   By Proposition \ref{max}, one has
   $$\max_{z\in M}(F(z,t)-u(z,t))\leq \max_{z\in M}(F(z,t_0)-u(z,t_0))=0,\quad t\geq t_0.$$
   It follows that
   \begin{align*}
     \max_{z\in M}F(z,t_0+a)&\leq \max_{z\in M}u(z,t_0+a)\\
     &=\max_{z\in M}\int_M H(z,w,a)F(w,t_0)dw\\
     &\leq C_a\int_M F(w,t_0)dw,
   \end{align*}
   where $C_a=\max_{M\times M}H(z,w,a)$. Fix $a$ and let $t_0\to \infty$, we conclude
   $$\max_{z\in M}F(z,t)\to 0\quad t\to\infty,$$
   which competes the proof.
   \end{enumerate}

\end{proof}

Inspired by  Proposition \ref{prop2} (3), we give the following definition of approximate geodesic-Einstein structure, which is similar as the approximate Hermitian-Einstein structure for a holomorphic vector bundle (see e.g. \cite[Section 4.5]{Ko3}).
\begin{defn}\label{defn1}
  We say that $L$ admits an approximate geodesic-Einstein structure, if for any given $\epsilon>0$, there exists an metric $\phi_{\epsilon}$ on $L$ such that
  $$\max_{ M}\left|\int_{\mc{X}/M}(tr_{\omega}c(\phi_{\epsilon})-\lambda)^2(\sqrt{-1}\p\b{\p}\phi_{\epsilon})^n\right|^{\frac{1}{2}}< \epsilon.$$
\end{defn}
By Proposition \ref{prop2} (3) and above definition, we obtain the following theorem.
\begin{thm}\label{thm10}
  Suppose that the geodesic-Einstein flow has a smooth solution for $0\leq t<+\infty$, then we have implications $(1)\Rightarrow (2)\Rightarrow (3)$ for the following statements:
  \begin{itemize}
  \item[(1)] the  functional $\mc{L}$ is bounded from below;
  \item[(2)]  $L$ admits an approximate geodesic-Einstein structure;
  \item[(3)] the pair $(\mc{X},L)$ is nonlinear semistable.
  \end{itemize}
\end{thm}
\begin{proof}
  By Proposition \ref{prop2} (3) and Definition \ref{defn1}, then $\mc{L}$ admits an approximate geodesic-Einstein structure. Now we begin to prove that a pair $(\mc{X}, L)$ is nonlinear semistable if $L$ admits an approximate geodesic-Einstein structure. In fact,
 \begin{align*}
   \lambda_{\mc{Y},L}
   &\geq \frac{(tr_{\omega}c(\phi)_{\mc{X}}[\omega]^mc_1(L)^{n'})[\mc{Y}]}{([\omega]^mc_1(L)^{n'})[\mc{Y}]}\\
   &=\lambda_{\mc{X},L}+\frac{((tr_{\omega}c(\phi)_{\mc{X}}-\lambda_{\mc{X},L})[\omega]^mc_1(L)^{n'})[\mc{Y}]}{([\omega]^mc_1(L)^{n'})[\mc{Y}]}\\
   &\geq \lambda_{\mc{X},L}-\frac{\epsilon}{(c_1(L)^{n'}[\mc{Y}/M])^{1/2}},
 \end{align*}
 where the last inequality holds by
 \begin{align*}
   &\quad \left| \int_{\mc{Y}/M}(tr_{\omega}c(\phi)_{\mc{X}}-\lambda_{\mc{X},L})c_1(L)^{n'}\right|\\
   &\leq \left| \int_{\mc{Y}/M}(tr_{\omega}c(\phi)_{\mc{X}}-\lambda_{\mc{X},L})^2c_1(L)^{n'}\right|^{1/2}(c_1(L)^{n'}[\mc{Y}/M])^{1/2}\\
   &\leq \left| \int_{\mc{X}/M}(tr_{\omega}c(\phi)_{\mc{X}}-\lambda_{\mc{X},L})^2c_1(L)^{n}\right|^{1/2}(c_1(L)^{n'}[\mc{Y}/M])^{1/2}\\
   &<\epsilon(c_1(L)^{n'}[\mc{Y}/M])^{1/2}.
 \end{align*}
 By taking $\epsilon\to 0$, we have $\lambda_{\mc{Y},L}\geq \lambda_{\mc{X},L}$, which completes the proof by Definition \ref{defn1}.
\end{proof}
\begin{rem}
	For the case of holomorphic vector bundle over compact K\"ahler manifold $M$,
	 Proposition \ref{prop2} and Theorem \ref{thm10} were proved in \cite[Proposition 6.9.1, Theorem 6.10.13]{Ko3}. In particular, if $M$ is projective, then $(1), (2), (3)$  are equivalent (see \cite[Theorem 6.10.13]{Ko3}), and he also conjectured that they should be equivalent in general whether $M$ is algebraic or not. Later, $(3)\Rightarrow (2)$ was proved in \cite{Jacob, LZ} if $M$ is K\"ahler. For a general pair ($\mc{X}$, $L$), one may ask that whether $(3)\Rightarrow (1)$ if $M$ is projective, and whether $(3)\Rightarrow(2)$ if $M$ is K\"ahler.
\end{rem}

\section{The case of $tr_{\omega}c(\phi)\geq 0$}\label{sec3}

In this section, we assume that there exists a metric $\phi$ on $L$ such that
\begin{align}
	tr_{\omega}c(\phi)\geq 0.
\end{align}
  For any given $\phi\in F^+(L)$ and the natural frame $\{\frac{\p}{\p z^{\alpha}},1\leq\alpha\leq m\}$ of $TM$, one sees that there is  canonical liftings $\{\frac{\delta}{\delta z^{\alpha}}, 1\leq \alpha\leq m\}$. Thus for  any vector $X=X^{\alpha}\frac{\p}{\p z^{\alpha}}|_y\in T_yM$ at a point $y\in M$, there is a canonical lifting
$$\wt{X}=X^{\alpha}\frac{\delta}{\delta z^{\alpha}},$$
which is a vector field on $\mc{X}_y=\pi^{-1}(y)$. We call the canonical lifting $\wt{X}$ is holomorphic if
\begin{align}
\b{\p}^V\wt{X}=\b{\p}^V	\left(X^{\alpha}\frac{\delta}{\delta z^{\alpha}}\right)=X^{\alpha}\b{\p}^V	\left(\frac{\delta}{\delta z^{\alpha}}\right)=X^{\alpha}\frac{\p}{\p v^i}(-\phi_{\alpha\b{l}}\phi^{\b{l}k})\delta v^i=0.
\end{align}

For any holomorphic vector bundle $E$ over $M$, the degree of $E$ is denoted
\begin{align}
\deg_{\omega}E=\int_M c_1(E)\wedge [\omega]^{m-1}.	
\end{align}

By using Berndtsson's curvature formula of direct image bundle, we obtain
\begin{thm}\label{thm1}
	If there exists a metric $\phi\in F^+(L)$ such that $tr_{\omega}c(\phi)\geq 0$, then $\deg_{\omega} \pi_*(L+K_{\mc{X}/M})=0$ if and only if $\lambda_{\mc{X},L}=0$  and the canonical lifting of any vector  is holomorphic. In particular, $\phi$ is a geodesic-Einstein metric on $L$.
\end{thm}
\begin{proof}
We first prove the last argument. If $tr_{\omega}c(\phi)\geq 0$, by the definition of $\lambda_{\mc{X},L}$ (see (\ref{lamda})), then
\begin{align}\label{1.19}
	     \lambda_{\mc{X},L}=\frac{2\pi m}{n+1}\frac{([\omega]^{m-1}c_1(L)^{n+1})[\mc{X}]}{([\omega]^mc_1(L)^n)[\mc{X}]}=\frac{\int_{\mc{X}/M}tr_{\omega}c(\phi)\omega^n_F}{\int_{\mc{X}/M}\omega^n_F}\geq 0,
\end{align}
the equality holds if and only if $tr_{\omega}c(\phi)=0$. Thus $\phi$ is a geodesic-Einstein metric on $L$ if $tr_{\omega}c(\phi)\geq 0$ and  $\lambda_{\mc{X},L}=0$.

Denote $E=\pi_*(L+K_{\mc{X}/M})$. By \cite[Theorem 1.2]{Bern4} and taking trace with respect to $\omega$, one has
\begin{align}\label{1.17}
\langle tr_{\omega}\Theta^{E}u,u\rangle=\int_{\mc{X}_y}tr_{\omega}c(\phi)|u|^2e^{-\phi}+g^{\alpha\b{\beta}}\langle(1+\Delta')^{-1}i_{\b{\p}^V(\frac{\delta}{\delta z^{\alpha}})}u,i_{\b{\p}^V(\frac{\delta}{\delta z^{\beta}})}u\rangle\sqrt{-1}dz^{\alpha}\wedge d\b{z}^{\beta}.
\end{align}
Combining the assumption $tr_{\omega}c(\phi)\geq 0$ shows that
\begin{align}\label{1.18}
\langle tr_{\omega}\Theta^{E}u,u\rangle\geq 0.	
\end{align}
By taking trace to $tr_{\omega}\Theta^E$ and using (\ref{1.18}), one gets
\begin{align}
	tr_E(tr_{\omega}\Theta^E)\geq 0.
\end{align}
Therefore,
\begin{align}
\begin{split}
\deg_{\omega}E &=\int_M c_1(E)\wedge \omega^{m-1}
=\frac{1}{m}\int_M tr_{\omega}c_1(E)\omega^m\\
&=\frac{1}{m}\int_M tr_{\omega}tr_E\Theta^E\omega^m
=\frac{1}{m}\int_M tr_Etr_{\omega}\Theta^E\omega^m\geq 0.
\end{split}
\end{align}
Moreover, by (\ref{1.19}), (\ref{1.17}) and (\ref{1.18}), $\deg_{\omega}E=0$ if and only if $tr_{\omega}c(\phi)=0$ and $\b{\p}^V(\frac{\delta}{\delta z^{\alpha}})=0$, which is equivalent to
\begin{align}
	 \lambda_{\mc{X},L}=\frac{\int_{\mc{X}/M}tr_{\omega}c(\phi)\omega^n_F}{\int_{\mc{X}/M}\omega^n_F}=0
\end{align}
and for any canonical lifting $\tilde{X}$, one has
\begin{align}
\b{\p}^V\wt{X}=X^{\alpha}\b{\p}^V	\left(\frac{\delta}{\delta z^{\alpha}}\right)=0,
\end{align}
i.e., $\wt{X}$ is holomorphic, which completes the proof.
\end{proof}

\begin{rem}
\begin{itemize}
\item[(1)] The above theorem is inspired by \cite[Theorem 2.4]{Bern3} for the case of semipositive line bundle $L$ over a local holomorphic fibration.
\item[(2)]	Theorem \ref{thm1} gives a sufficient condition for the existence of geodesic-Einstein metric for the case $\lambda_{\mc{X},L}=0$. For a geodesic-Einstein equation $tr_{\omega}c(\phi)=\lambda_{\mc{X},L}$
	with $\lambda_{\mc{X},L}\neq 0$, we can reduce it to the case of $tr_{\omega}c(\phi)=0$.
In fact, for any line bundle $L'\to M$ with $\deg_{\omega}L'\neq 0$, there exists a Hermitian-Einstein metric $\phi'$ on $M$ (see \cite[Proposition 4.1.4 and Proposition 4.2.4]{Ko3} or \cite[Chapter 1, (1.4) Remark (i)]{Siu}), i.e. $tr_{\omega}c_1(L',\phi')=\lambda'\neq 0$. Since both $\lambda_{\mc{X},L}$ and $\lambda'$ are rational, so there exist integers $a>0,b$ such that $a\lambda_{\mc{X},L}+b\lambda'=0$, thus
$$tr_{\omega}c(\psi)=a\lambda_{\mc{X},L}+b\lambda'=0.$$  	
	Here $\psi=a\phi+b\pi^*\phi'$ is the weight of the line bundle $aL+b\pi^*L'$.
\item[(3)] By a direct calculation, a homogenous geodesic-Einstein equation $tr_{\omega}c(\phi)=0$ is equivalent to
\begin{align}\label{2.2}
(\sqrt{-1}\p\b{\p}\phi)^{n+1}\wedge \omega^{m-1}=0.	
\end{align}
Moreover,  if one considers the homogenous geodesic curvature equation $c(\phi)=0$, by Lemma \ref{lemma0}, it is equivalent to
\begin{align}\label{2.3}
	(\sqrt{-1}\p\b{\p}\phi)^{n+1}=0.
\end{align}
In our next paper, we will try to study the above two equations (\ref{2.2}) and (\ref{2.3}).
	\end{itemize}
	
\end{rem}

\section{$S$-class and $C$-class}\label{sec4}

In this section, we will define the total $S$-class $S(L)$ and the total $C$-class $C(L)$ for a relative ample line bundle $L$ and discuss some inequalities and positivity.

Let $\pi:E\to M$ be a holomorphic vector bundle of rank $r$ over $M$,  then there is a canonical pair $(P(E), \mc{O}_{P(E)}(1))$.  The Segre classes are defined by
\begin{align}
s_i(E)=\pi_*((c_1(\mc{O}_{P(E)}(1)))^{r-1+i})\in H^{2i}(M,\mb{Z}),\quad 0\leq i\leq m=\dim M.	
\end{align}
Then the total Segre class is given by
\begin{align}
s(E)=\sum_{i=0}^{m}	s_i(E).
\end{align}
The total Chern class and Chern classes of $E$ can be defined by
\begin{align}
c(E)=\frac{1}{s(E)}, \quad c(E)=\sum_{i=0}^m c_i(E),\quad c_i(E)\in H^{2i}(M,\mb{Z}),	
\end{align}
(see e.g. \cite[Chapter 3]{Fulton} or \cite[Section 20]{Bott}).

Inspired by the above construction, for a general pair $(\mc{X},L)$, $\pi:\mc{X}\to M$, we define the total $S$-class and the $S$-classes of $L$ by
\begin{align}
 S(L)=\sum_{i=0}^m S_i(L),\quad S_i(L)=\pi_{*}((c_1(L))^{n+i})\in H^{2i}(M,\mb{Z}), \quad 0\leq i\leq m,
\end{align}
and the total $C$-class and the $C$-classes are defined by
\begin{align}\label{2.1}
C(L)=\frac{1}{S(L)},\quad 	C(L)=\sum_{i=0}^m C_i(L),\quad C_i(L)\in H^{2i}(M,\mb{Q}).
\end{align}
By above definition, one has
\begin{align}\label{1.10}
C_0(L)=\frac{1}{S_0(L)},\quad C_1(L)=-\frac{S_1(L)}{S_0(L)^2},\quad C_2(L)=\frac{S_1(L)^2-S_0(L)S_2(L)}{S_0(L)^3},
\end{align}
where $S_0(L)=\int_{\mc{X}/M}c_1(L)^n\in\mb{N}_+$.


\subsection{Some inequalities}
In this subsection, we assume that there exists a geodesic-Einstein metric on $L$ and discuss some inequalities in terms of $S$-class and $C$-class.

For any smooth metric $\phi$ on $L$, it induces a natural representation of $S_i(L)$ by
\begin{align}\label{S form}
S_i(L,\phi)=\pi_*((c_1(L,\phi))^{n+i})=\int_{\mc{X}/M}\left(\frac{\sqrt{-1}}{2\pi}\p\b{\p}\phi	\right)^{n+i}
\end{align}
 and set $S(L,\phi)=\sum_{i=0}^m S_i(L,\phi)$. By the relation (\ref{2.1}), we obtain the representations of $C(L)$, $C_i(L)$ by
 \begin{align}
 C(L,\phi)=\frac{1}{S(L,\phi)}=\sum_{i=1}^m C_{i}(L,\phi).	
 \end{align}
 We call $S_i(L, \phi)$ and $C_i(L,\phi)$ the $S$-form and $C$-form, respectively.
If $L$ admits a geodesic-Einstein metric, then we obtain
\begin{thm}\label{thm3}
  If $\phi$ is a geodesic-Einstein metric on $L$, i.e., $tr_{\omega}c(\phi)=\lambda_{\mc{X},L}$, then
  \begin{align}\label{1.7}
    S_2(L,\phi)\wedge \omega^{m-2}\leq \frac{(n+1)(n+2)}{8\pi^2 m^2}\lambda^2_{\mc{X},L}S_0(L)\omega^m,
  \end{align}
  the equality holds if and only if $c(\phi)=\frac{\lambda_{\mc{X},L}}{m}\omega$. In particular,
  \begin{align}\label{1.8}
  \int_{M} S_2(L)\wedge [\omega]^{m-2}\leq 	\frac{(n+1)(n+2)}{8\pi^2 m^2}\lambda^2_{\mc{X},L}S_0(L)\int_{M}[\omega]^m.
  \end{align}

\end{thm}
\begin{proof}
   By Lemma \ref{lemma0}, the first Chern class of $L$ is represented by
  $$\frac{\sqrt{-1}}{2\pi}\p\b{\p}\phi=\frac{1}{2\pi}c(\phi)
  +\frac{\sqrt{-1}}{2\pi}\phi_{i\b{j}}\delta v^i\wedge \delta \b{v}^j.$$
  Denote $\omega_{F}:=\frac{\sqrt{-1}}{2\pi}\phi_{i\b{j}}\delta v^i\wedge \delta \b{v}^j$ and by (\ref{S form}), then
  \begin{align}\label{1.9}
  \begin{split}
    S_2(L,\phi)\wedge \omega^{m-2}&=\pi_*\left(\frac{\sqrt{-1}}{2\pi}\p\b{\p}\phi\right)^{n+2}\wedge \omega^{m-2}\\
    &=\frac{(n+1)(n+2)}{8\pi^2}\pi_*(c(\phi)^2\omega^n_F)\wedge\omega^{m-2}\\
    &=\frac{(n+1)(n+2)}{8\pi^2m(m-1)}\pi_*\left(((tr_{\omega}c(\phi))^2-|c(\phi)|^2_{\omega})\omega^{n}_F\right)\wedge\omega^m\\
    &\leq \frac{(n+1)(n+2)}{8\pi^2m(m-1)}\pi_*\left(((tr_{\omega}c(\phi))^2-\frac{1}{m}(tr_{\omega}c(\phi))^2)\omega^{n}_F\right)\wedge \omega^m\\
    &=\frac{(n+1)(n+2)}{8\pi^2 m^2}\lambda^2_{\mc{X},L}S_0(L)\omega^m,
    \end{split}
  \end{align}
  where the third equality follows from the following formula,
  \begin{align}\label{1.11}
m(m-1)\alpha\wedge\alpha\wedge \omega^{m-2}=\left((tr_{\omega}\alpha)^2-|\alpha|^2_{\omega}\right)\omega^m
\end{align}
for any real $(1,1)$-form $\alpha$ (see e.g. \cite[Lemma 4.7]{Sz}).
  The fourth equality in (\ref{1.9}) holds since
  \begin{align}
    |c(\phi)|^2_{\omega}:=c(\phi)_{\alpha\b{\beta}}c(\phi)_{\gamma\b{\delta}}g^{\alpha\b{\delta}}g^{\gamma\b{\beta}}
  \end{align}
  and $|c(\phi)|^2_{\omega}\geq \frac{1}{m}(tr_{\omega}c(\phi))^2$.
  Moreover, the equality holds if and only if
  \begin{align}
    c(\phi)_{\alpha\b{\beta}}=\frac{\lambda_{\mc{X},L}}{m}g_{\alpha\b{\beta}},
  \end{align}
  that is, $c(\phi)=\frac{\lambda_{\mc{X},L}}{m}\omega$. By integrating the both sides of (\ref{1.7}), we conclude (\ref{1.8}).
\end{proof}

In terms of $C$-classes, we have the following Kobayashi-L\"ubke type inequality.
\begin{thm}\label{thm4}
  If $\phi$ is a geodesic-Einstein metric on $L$, then
  \begin{align}\label{1.15}
  (nC_1(L,\phi)^2-2(n+1)C_0(L)C_2(L,\phi))\wedge \omega^{m-2}\leq 0,
  \end{align}
  the equality holds if and only if $c(\phi)=\frac{2\pi}{(n+1)S_0(L)}S_1(L,\phi)$. In particular,
  \begin{align}\label{1.14}
  \int_M (nC_1(L)^2-2(n+1)C_0(L)C_2(L))\wedge [\omega]^{m-2}\leq 0.	
  \end{align}
  \end{thm}
  \begin{proof}
  By (\ref{1.10}) and (\ref{1.11}), one has
  \begin{align}\label{1.12}
  	\begin{split}
  		C_1(L,\phi)^2\wedge \omega^{m-2}&=\frac{1}{S_0(L)^4}S_1(L,\phi)^2\wedge \omega^{m-2}\\
  		&=\frac{1}{S_0(L)^4}\frac{1}{m(m-1)}\left((tr_{\omega}S_1(L,\phi))^2-|S_1(L,\phi)|^2_{\omega}\right)\omega^m\\
  		&=\frac{1}{S_0(L)^4}\frac{1}{m(m-1)}\left(((n+1)\lambda_{\mc{X},L}S_0(L))^2-|S_1(L,\phi)|^2_{\omega}\right)\omega^m,
  	\end{split}
  \end{align}
and
\begin{align}\label{1.13}
\begin{split}
	C_2(L,\phi)\wedge \omega^{m-2}&=\frac{S_1(L,\phi)^2-S_0(L)S_2(L,\phi)}{S_0(L)^3}\wedge \omega^{m-2}\\
	&=S_0(L)C_1(L,\phi)^2\wedge \omega^{m-2}-\frac{1}{S_0(L)^2}S_2(L,\phi)\wedge \omega^{m-2}\\
	&=S_0(L)C_1(L,\phi)^2\wedge \omega^{m-2}-\frac{1}{S_0(L)^2}\frac{(n+1)(n+2)}{8\pi^2m(m-1)}(\lambda^2_{\mc{X},L}S_0(L)-\pi_*(|c(\phi)|^2_{\omega}\omega^n_{F}))\omega^m.
\end{split}	
\end{align}
From (\ref{1.12}) and (\ref{1.13}), we have
\begin{align}
\begin{split}
	&\quad (nC_1(L,\phi)^2-2(n+1)C_0(L)C_2(L,\phi))\wedge\omega^{m-2}\\
	 &=\left(\frac{1}{S_0(L)^4}\frac{n+2}{m(m-1)}|S_1(L,\phi)|^2_{\omega}-\frac{1}{S_0(L)^3}\frac{(n+2)(n+1)^2}{4\pi^2m(m-1)}\pi_*(|c(\phi)|^2_{\omega}\omega^n_F)\right)\omega^m\\
	 &=\frac{1}{S_0(L)^4}\frac{(n+2)(n+1)^2}{4\pi^2m(m-1)}\left(|\pi_*(c(\phi)\omega^n_F)|^2_{\omega}-S_0(L)\pi_*(|c(\phi)|^2_{\omega}\omega^n_F)\right)\omega^m.
\end{split}
\end{align}
 For any $p\in M$ and taking the normal coordinate system near $p$ with $g_{\alpha\b{\beta}}=\delta_{\alpha\b{\beta}}$, then
  \begin{align*}
    &\quad |\pi_*(c(\phi)\omega^n_F)|^2_{\omega}-S_0(L)\pi_*(|c(\phi)|^2_{\omega}\omega^n_F)\\
    &=\sum_{\alpha,\beta}\pi_*(c(\phi)_{\alpha\b{\beta}}\omega^n_{F})\pi_*(\o{c(\phi)_{\alpha\b{\beta}}}\omega^{n}_F)-\pi_*(\omega^n_F)\pi_*(\sum_{\alpha,\beta}|c(\phi)_{\alpha\b{\beta}}|^2\omega^n_F)\\
    &\leq \sum_{\alpha,\beta}\left((\pi_*(|c(\phi)_{\alpha\b{\beta}}|\omega^n_F))^2-\pi_*(\omega^n_F)\pi_*(|c(\phi)_{\alpha\b{\beta}}|^2\omega^n_F)\right)\leq 0,
  \end{align*}
  where the last inequality follows from Cauchy-Schwarz inequality. Thus,
  \begin{align}
  	(nC_1(L,\phi)^2-2(n+1)C_0(L)C_2(L,\phi))\wedge\omega^{m-2}\leq 0,
  \end{align}
  which proves (\ref{1.15}).
  Moreover, the equality holds  if and only if $(c(\phi)_{\alpha\beta})$ is constant along each fiber by Cauchy-Schwarz inequality, which is equivalent to
  $$c(\phi)=\frac{\int_{\mc{X}/M}c(\phi)\omega^n_F}{\int_{\mc{X}/M}\omega^n_F}=\frac{2\pi}{(n+1)S_0(L)}S_1(L,\phi).$$
  By taking integrate the both sides of (\ref{1.15}) over $M$ , we obtain the  topological inequality (\ref{1.14}).
\end{proof}
\begin{rem}
For the case of holomorphic vector bundle $E\to M$, the geodesic-Einstein  metric is equivalent to Finsler-Einstein metric (see \cite[Lemma 3.6]{Feng}). With the assumption of existence of Finsler-Einstein metric, Theorem \ref{thm3} and Theorem \ref{thm4} were proved in \cite[Theorem 3.7, 3.8]{FLW}. If $h$ is a Hermitian-Einstein metric on $E$, then Theorem \ref{thm4} is reduced to the classical Kobayashi-L\"ubke inequality \cite[Theorem 4.4.7]{Ko3} (see also \cite[Chapter 1, (1.8)]{Siu}), while Theorem \ref{thm3} is exactly the \cite[Theorem 1.2]{Diver}.
 \end{rem}

\subsection{Positivity of classes}

In this subsection, we will discuss the positivity of $S$-classes and $C$-classes.

Recall a smooth $(p,p)$-form $\Phi$ on a complex manifold $M$ is positive if for any $y\in M$ and any linearly independent $(1,0)$-type tangent vectors $v_1,v_2,\cdots, v_p$ at $y$, it holds that
\begin{align}\label{2.111}
(-\sqrt{-1})^{p^2}\Phi(v_1,v_2,\cdots, v_p,\bar v_1,\bar v_2,\cdots, \bar v_p)>0.
\end{align}

\begin{prop}\label{t6.2} If $\sqrt{-1}\p\b{\p}\phi>0$,  then the $k$-th  $S$-form $S_{k}(L,\phi)$ is a positive $(k,k)$-form for any $0\leq k\leq m$.
In particular, the class $S_k(L)$ can be represented by a positive $(k,k)$-form if $L$ is ample.
\end{prop}
\begin{proof}
By Lemma \ref{lemma0}, $\sqrt{-1}\p\b{\p}\phi>0$ is equivalent to
\begin{align}
c(\phi)>0	
\end{align}
on horizontal directions.
 For any $(z_0,v_0)\in \mc{X}$, there exists a basis $\{\psi^1,\cdots,\psi^n\}$  such that
\begin{align}\label{2.112}
c(\phi)={\sqrt{-1}}\sum^n_{\alpha=1}\psi^\alpha\wedge\bar\psi^\alpha,
\end{align}
and so for $k\geq 1$,
\begin{align}\label{2.113}
(-{\sqrt{-1}})^{k^2}c(\phi)^k=k!\sum_{1\leq\alpha_1<\cdots<\alpha_k\leq n}\psi^{\alpha_1}\wedge\cdots\wedge\psi^{\alpha_k}\wedge\bar\psi^{\alpha_1}\wedge\cdots\wedge\bar\psi^{\alpha_k}.
\end{align}
Hence for any independent horizontal vectors $X_1,\cdots,X_k$ at $(z_0,v_0)$, one has
\begin{align}\label{2.114}
\begin{split}
&(-{\sqrt{-1}})^{k^2}c(\phi)^k(X_1,\cdots,X_k,\o{X_1},\cdots,\o{X_k})\\
&=k!\sum_{1\leq\alpha_1<\cdots<\alpha_k\leq n}\psi^{\alpha_1}\wedge\cdots\wedge\psi^{\alpha_k}\wedge\bar\psi^{\alpha_1}\wedge\cdots\wedge\bar\psi^{\alpha_k}
(X_1,\cdots,X_k,\o{ X_1},\cdots,\o{X_k})\\
&=k!\sum_{1\leq\alpha_1<\cdots<\alpha_k\leq n}|\psi^{\alpha_1}\wedge\cdots\wedge\psi^{\alpha_k}(X_1,\cdots,X_k)|^2> 0.
\end{split}
\end{align}
On the other hand, by Lemma \ref{lemma0}, by (\ref{2.114}), one has
\begin{align}\label{2.115}
S_{k}(L,\phi)=\int_{\mc{X}/M}(\sqrt{-1}\p\b{\p}\phi)^{n+k}=\frac{1}{(2\pi)^{k}}\binom{n+k}{k}\int_{\mc{X}/M}c(\phi)^{k}\omega^{n}_{F}.
\end{align}
For any $z_0\in M$ and any linearly independent $(1,0)$-type tangent vectors $Y_1,\cdots,Y_{k}$ in $T_{z_0}M$, one has
\begin{align*}
&(-\sqrt{-1})^{k^2}S_k(L,\phi)(Y_1,\cdots,Y_{k},\bar Y_1,\cdots,\bar Y_k)\\
&=\frac{1}{(2\pi)^{k}}\binom{n+k}{k}\int_{\mc{X}_{z_0}}
(-\sqrt{-1})^{k^2}c(\phi)^{k}(Y^h_1,\cdots,Y^h_{k},\o{ Y^h_1},\cdots,\o{ Y^h_{k}})\omega^{n}_{F}>0,
\end{align*}
where $Y^h$ denote the horizontal lifting along the fibre $P(E_{z_0})$ of a vector $Y\in T_{z_0}M$. Thus we conclude $S_k(L,\phi)$ is a positive $(k,k)$-form. Since $[S_k(L,\phi)]=S_k(L)$, so the class $S_k(L)$ can be represented by a positive $(k,k)$-form  if $L$ is ample.
\end{proof}

\begin{cor}\label{cor11}
Let $M$ be a compact complex surface and $L$ be an ample line bundle over $\mc{X}$, $\pi:\mc{X}\to M$. If moreover, $L$ admits a geodesic-Einstein metric, then
\begin{align*}
\int_M C_2(L)>0.	
\end{align*}
\end{cor}
\begin{proof}
 By Theorem \ref{thm4}, (\ref{1.10}) and noting $\dim M=2$, then
\begin{align}\label{1.16}
\begin{split}
\int_M C_2(L) &\geq \int_M \frac{nC_1(L)^2}{2(n+1)C_0(L)}\\
&=\frac{n}{2(n+1)S_0(L)^3}\int_M S_1(L)^2>0,
\end{split}
\end{align}
where the last inequality follows from Theorem \ref{t6.2}.
\end{proof}

\begin{rem}
In terms of complex Finsler vector bundles, Theorem \ref{t6.2} was proved in \cite[Theorem 2.8]{FLW} with the assumption of positive or negative Kobayashi curvature. On the other hand, if one considers an ample vector bundle $E$, which is equivalent to the pair $(P(E^*), \mc{O}_{P(E^*)}(1))$ with positive line bundle $\mc{O}_{P(E^*)}(1)$, by \cite[Theorem 2.5]{BG}, one has $\int_M c_m(E)>0$. Since the $C$-classes $C_i(L)$ can be viewed as a generalization of Chern classes of a holomorphic vector bundle, so it is natural to ask  whether $(-1)^m\int_M C_m(L)>0$ if $L$ is ample.
\end{rem}

\section{Some examples}\label{sec example}
In this section, we will discuss some examples on the geodesic-Einstein metrics.
\begin{ex}[Product manifolds]\label{exam1}
Let $X$ and $M$ be two compact complex manifolds and consider the holomorphic fibration $\pi_1: M\times X\to M$. For any line bundle $L_1$ over $M$ and any ample line bundle $L_2$ over $X$, then the line bundle $L:=\pi_1^*L_1+\pi_2^*L_2$ over $X\times M$ admits a geodesic-Einstein metric with respect to any given K\"ahler metric $\omega$ on $M$, where $\pi_2: M\times X\to X$. In fact, for any K\"ahler metric $\omega$ on $M$, we take a Hermitian-Einstein metric $\varphi_1$ on $L_1$, i.e. $tr_{\omega}(\sqrt{-1}\p\b{\p}\varphi_1)=\text{constant}$, and take a metric $\varphi_2$ on $L_2$ such that $\sqrt{-1}\p\b{\p}\varphi_2>0$. Therefore, $\varphi=\pi^*_1\varphi_1+\pi^*\varphi_2$ is a metric on $L$ and its curvature is $\p\b{\p}\varphi=\pi_1^*\p\b{\p}\varphi_1+\pi^*_2\p\b{\p}\varphi_2$. So $L$ is a relative ample line bundle and
\begin{align*}
tr_{\omega}c(\phi)=tr_{\omega}(\sqrt{-1}\p\b{\p}\varphi_1)=\text{constant}. 	
\end{align*}
\end{ex}

\begin{ex}[Ruled manifolds]
	An algebraic manifold $\mc{X}$ is said to be a {\it ruled manifold} if $\mc{X}$ is a holomorphic $\mb{P}^r$-bundle with structure group $PGL(r+1,\mb{C})=GL(r+1,\mb{C})/\mb{C}^*$ (see e.g. \cite[Section 4.2]{Aikou}). By \cite[Proposition 4.3]{Aikou}, every ruled manifold $\mc{X}$ over a compact Riemann surface
$M$ is holomorphically isomorphic to $P(E)$ for some holomorphic vector bundle $E\to M$ of $\text{rank}(E)=r+1$. Such a bundle $E$ is uniquely determined up to tensor product with a holomorphic line bundle. Since $\text{Pic}(P(E))=\text{Pic}(M)\oplus \mb{Z}\mc{O}_{P(E)}(1)$, so any line bundle $L$ over $P(E)$ is the form $L:=\pi^*L_1+k\mc{O}_{P(E)}(1)$ for some line bundle $L_1$ over $M$ and $k\in \mb{Z}$. Moreover, $L$ is relative ample if and only if $k>0$.
	If there exists a Hermitian-Einstein metric on $E$, then $L$ can admit a geodesic-Einstein metric. In fact, by a direct calculation, any Hermitian-Einstein metric on $E$ induces a natural geodesic-Einstein metric on $\mc{O}_{P(E)}(1)$, so is $k\mc{O}_{P(E)}(1)$. Similar as Example \ref{exam1}, the geodesic-Einstein metric on $k\mc{O}_{P(E)}(1)$ and the Hermitian-Einstein metric on $L_1$ give a geodesic-Einstein metric on $L$.
\end{ex}

\begin{ex}[Geodesic curve]
	From Definition \ref{GeoEin}, one can allow $M$ is a non-compact manifold. Now we assume that $M$ is a Riemann surface with boundary and $\mc{X}=X\times M$ for some compact complex manifold $X$, the line bundle $L\to \mc{X}$ is taken to be the pullback of some ample line bundle over $X$. Then $tr_{\omega}c(\phi)=0$ is equivalent to $c(\phi)=0$, which is also equivalent to
	$(\sqrt{-1}\p\b{\p}\phi)^{n+1}=0$ (see \cite{Don3, Semmes}). As seen in \cite[Section 2.3]{Chen} or \cite[Section 3]{Mab}, one can define a Riemannian metric on the space of K\"ahler potentials  in a fixed K\"ahler class, which is an infinite dimensional manifold and equipped a $L^2$-norm, then its geodesic equation is exactly the equation $c(\phi)=0$.  By \cite[Theorem 3]{Chen}, the following equation
	\begin{align*}
	\begin{cases}
		&(\sqrt{-1}\p\b{\p}\phi)^{n+1}=0\quad \text{in}\quad \mc{X}\\
		&\phi=\phi_0\quad \text{in}\quad \p\mc{X}
	\end{cases}
	\end{align*}
has a $C^{1,1}$-solution for any  metric $\phi_0$ in $F^+(L|_{\p\mc{X}})$.
\end{ex}

\begin{ex}[Complex quotient equation]
	From (\ref{geolam}), the geodesic-Einstein equation $tr_{\omega}c(\phi)=\lambda$ is equivalent to
	\begin{align}\label{geoe1}
		c\omega^m\wedge(\sqrt{-1}\p\b{\p}\phi)^n=\omega^{m-1}\wedge (\sqrt{-1}\p\b{\p}\phi)^{n+1},
	\end{align}
where $c=\frac{\lambda(n+1)}{m}$. If moreover, we assume that $L$ is an ample line bundle,  and one may consider a geodesic-Einstein metric $\phi$ with positive curvature, i.e. $tr_{\omega}c(\phi)=\lambda$ and $\sqrt{-1}\p\b{\p}\phi>0$.
Let $\phi_0$ be a metric on $L$ with $\sqrt{-1}\p\b{\p}\phi_0>0$, then $\omega_{\epsilon}=\omega+\epsilon\sqrt{-1}\p\b{\p}\phi_0$ is a K\"ahler metric on
	$\mc{X}$. For any $\epsilon>0$, we consider the following equation
	\begin{align}\label{quotient}
		c_{\epsilon}\omega_{\epsilon}^m\wedge(\sqrt{-1}\p\b{\p}\phi)^n=\omega_{\epsilon}^{m-1}\wedge (\sqrt{-1}\p\b{\p}\phi)^{n+1},
	\end{align}
	where $c_{\epsilon}=\int_{\mc{X}}\omega_{\epsilon}^{m-1}\wedge (\sqrt{-1}\p\b{\p}\phi)^{n+1}/\int_{\mc{X}}\omega_{\epsilon}^m\wedge(\sqrt{-1}\p\b{\p}\phi)^n$.
The above equation is studied in \cite{Phong, Sun} by using the following parabolic flow
\begin{align}\label{para1}
\begin{split}
\frac{\p u}{\p t}=\log \frac{\omega_{\epsilon}^m\wedge (\sqrt{-1}\p\b{\p}\phi)^{n}}{\omega_{\epsilon}^{m-1}\wedge (\sqrt{-1}\p\b{\p}\phi)^{n+1}}+\log c_{\epsilon},
\end{split}
\end{align}
where $u=\phi-\phi_0$. From \cite[Theorem 1.2]{Sun} or \cite[Corollary 7]{Phong}, if there is a $\mc{C}$-subsolution to (\ref{para1}) (see \cite[Definition 1.1]{Sun}), then there exists a long time solution $u$ to (\ref{para1}). Moreover, the normalization $\hat{u}$ (see \cite[(2.34)]{Sun}) of $u$ is $C^{\infty}$ convergent to a smooth solution $\hat{u}_{\infty}$.  Thus $\sqrt{-1}\p\b{\p}\phi=\sqrt{-1}\p\b{\p}(\phi_0+\hat{u}_{\infty})$ solves (\ref{quotient}). In particular, if $M$ is a Riemann surface, then (\ref{quotient}) becomes
\begin{align}\label{quotient1}
		c_{\epsilon}\omega_{\epsilon}\wedge(\sqrt{-1}\p\b{\p}\phi)^{n}= (\sqrt{-1}\p\b{\p}\phi)^{n+1},
	\end{align}
	which is the Euler equation of $J$-functional (see \cite{Don4}). From \cite[Theorem 1.1]{Song}, there exists a solution to (\ref{quotient1}) if and only if there exists a metric $\sqrt{-1}\p\b{\p}\phi'>0$ such that
	\begin{align}
	((n+1)\sqrt{-1}\p\b{\p}\phi'-nc_{\epsilon}\omega_{\epsilon})	\wedge (\sqrt{-1}\p\b{\p}\phi')^{n-1}>0.
	\end{align}
Comparing with the equation (\ref{quotient}) and (\ref{quotient1}), the  metric $\omega$ in geodesic-Einstein equation (\ref{geoe1}) is degenerate along the fibers of $\mc{X}$. Thus, in order to solve the geodesic-Einstein equation (\ref{geoe1}), one may consider the limits of the solutions of (\ref{quotient}), (\ref{quotient1})   as $\epsilon\to 0$.
\end{ex}

\begin{ex}[Calabi-Yau family]
	From (\ref{geoe1}), one can define a geodesic-Einstein metric for any holomorphic line bundle (need not relative ample). More precisely, for any line bundle $L$ over $\mc{X}$, a smooth metric $\phi$ on $L$ is called geodesic-Einstein if
	\begin{align}\label{geoe2}
		c\omega^m\wedge(\sqrt{-1}\p\b{\p}\phi)^n=\omega^{m-1}\wedge (\sqrt{-1}\p\b{\p}\phi)^{n+1}
	\end{align}
	for some constant $c$.
	
	Following \cite{Braun}, let $\pi:\mc{X}\to M$ be a holomorphic, polarized family of Calabi-Yau manifolds $\mc{X}_s=\pi^{-1}(s)$, $s\in M$, i.e.
compact manifolds with $c_1(\mc{X}_s)=0$, equipped with Ricci flat K\"ahler forms $\omega_{\mc{X}_s}$. The relative volume
form $\omega^n_{\mc{X}/M}=g dV$ induces a hermitian metric $g^{-1}$ on the relative canonical bundle $K_{\mc{X}/M}$. By \cite[Formula (1)]{Braun}, one has
	\begin{align}\label{Calabi-Yau}
	2\pi c_1(K_{\mc{X}/M}, g^{-1})=\sqrt{-1}\p\b{\p}\log g=\frac{1}{\text{Vol}\mc{X}_s}\pi^*\omega^{WP}.
	\end{align}
Now we take $L=K_{\mc{X}/M}$, $\phi=\log g$ and for any K\"ahler metric $\omega$ on $M$. By (\ref{Calabi-Yau}), one easily sees that the metric $\phi$ satisfies the equation (\ref{geoe2}), i.e., $\phi$ is a geodesic-Einstein metric on $L$.
	\end{ex}

\section{Appendix: Hermitian-Einstein metrics on quasi-vector bundles (By Xu Wang)}\label{Sec HE}

In this section, we will prove that the geodesic-Einstein metric is equivalent to a Hermitian-Einstein metric on a quasi-vector bundle. For more details on quasi-vector bundle, one can refer to \cite{Bern5, Wang}.

Following \cite{Bern5, Wang}, we shall use the following setup:
\begin{itemize}
\item[(1)] $\pi: \mc{X}\to B$ is a proper holomorphic submersion from a complex manifold $\mc{X}$ to another complex manifold $B$, each fiber $X_t:=\pi^{-1}(t)$ is an $n$-dimensional compact complex manifold;
\item[(2)] $E$ is a holomorphic vector bundle over $\mc{X}$, $E_t:=E|_{X_t}$;
\item[(3)] $\omega$ is a d-closed $(1,1)$-form on $\mc{X}$ and is positive on each fiber, $\omega^t:=\omega|_{X_t}$;
\item[(4)] $h_E$ is a smooth Hermitian metric on $E$, $h_{E_t}:=h_E|_{E_t}$.
\end{itemize}
\begin{defn}[\cite{Bern5}]
	Let $V:=\{V_t\}_{t\in B}$ be a family of $\mb{C}$-vector spaces over $B$. Let $\Gamma$ be a $C^{\infty}(B)$-submodule of the space of all sections of $V$. We call $\Gamma$ a smooth quasi-vector bundle structure on $V$ if each vector of the fiber $V_t$ extends to a section in $\Gamma$ locally near $t$.
\end{defn}
For each $t\in B$, let us denote by $\mc{A}^{p,q}(E_t)$ the space of smooth $E_t$-valued $(p,q)$-forms on $X_t$. Consider
$$\mc{A}^{p,q}:=\{\mc{A}^{p,q}(E_t)\}_{t\in B},$$
and denote by $\mc{A}^{p,q}(E)$ the space of smooth $E$-valued $(p,q)$-forms on $\mc{X}$. Let us define
$$\Gamma^{p,q}:=\{u: t\mapsto u^t\in\mc{A}^{p,q}(E_t): \exists\, {\bf u}\in \mc{A}^{p,q}(E), {\bf u}|_{X_t}=u^t,\forall t\in B\}.$$
We call ${\bf u}$ above a smoooth representative of $u\in \Gamma^{p,q}$, each $\Gamma^{p,q}$ defines a quasi-vector bundle structure on $\mc{A}^{p,q}$. Denote
$$(\mc{A},\Gamma):=\oplus_{k=0}^{2n}(\mc{A}^k,\Gamma^k),\quad (\mc{A}^k,\Gamma^k):=\oplus_{p+q=k}(\mc{A}^{p,q}, \Gamma^{p,q}).$$
\begin{defn}
The Lie-derivative connection, say $\n^{\mc{A}}$, on $(\mc{A},\Gamma)$ is defined as follows:
$$\n^{\mc{A}}u:=\sum dt^j\otimes [d^E,\delta_{V_j}]{\bf u}+\sum d\o{t}^j\otimes [d^E,\delta_{\bar{V}_j}]{\bf u},\quad u\in \Gamma,$$
where $d^E:=\b{\p}+\p^E$ denotes the Chern connection on $(E, h_E)$ and each $V_j$ is the horizontal lift of $\p/\p t^j$ with respect to $\omega$. 	
\end{defn}
Let $D^{\mc{A}}$ denote the connection on each $(V^{p,q},\Gamma^{p,q})$ induced by $\n^{\mc{A}}$,
\begin{align}\begin{split}D^{\mc{A}}u:&=\sum dt^j\otimes [\p^E,\delta_{V_j}]{\bf u}+\sum d\o{t}^j\otimes [\o{\p},\delta_{\o{V}_j}]{\bf u},\quad u \in \Gamma^{p,q}\\
&=\n^A-\left(\sum dt^j\otimes \kappa_j+\sum d\b{t}^j\otimes \kappa_{\b{j}}\right),\end{split}\end{align}
where the non-cohomological Kodaira-Spencer actions (see \cite[Definition 5.6]{Wang})
$$\kappa_j u:=[\b{\p},\delta_{V_j}]{\bf u},\quad \kappa_{\b{j}}u:=[\p^E,\delta_{\o{V}_j}]{\bf u}.$$
Then $D^{\mc{A}}$ defines a Chern connection on each $(\mc{A}^{p,q},\Gamma^{p,q})$ (see e.g. \cite[Theorem 5.6]{Wang}).
The curvature of the Chern connection $D^{\mc{A}}$ is given by
\begin{align}\label{Curvature of D}
	(D^{\mc{A}})^2=(\n^{\mc{A}})^2-\sum [\kappa_j, \kappa_{\b{k}}]dt^j\wedge d\b{t}^k,
\end{align}
see \cite[(5.2)]{Wang}. Here the curvature of Lie-derivative connection is
\begin{align}
\begin{split}
	(\n^{\mc{A}})^2 &=\sum [[d^E,\delta_{V_j}],[d^E,\delta_{\b{V}_k}]]dt^j\wedge d\b{t}^k.
	\end{split}
\end{align}
The $L^2$-metric on each $\mc{A}^{p,q}$ is defined by
\begin{align}\label{L2-metric}
(u,v)=\int_{X_t}\langle u, v\rangle\frac{\omega^n}{n!}.	
\end{align}
Here $\langle \cdot, \cdot\rangle$ denotes the point-wise inner on $\mc{A}^{p,q}$ with respect to $\omega|_{X_t}$ and $h_{E_t}$.
\begin{defn}\label{defn of HE}
The $L^2$-metric (\ref{L2-metric}) on  $\mc{A}^{p,q}$ is called Hermitian-Einstein with respect to a Hermitian metric $\omega_B=\sqrt{-1}g_{i\b{j}}dt^i\wedge d\b{t}^j$ if
\begin{align}
\Lambda_{\omega_B}(D^{\mc{A}})^2=\lambda \text{Id}	
\end{align}
for some constant $\lambda$.
\end{defn}
Now we take $p=q=0$ and $E$ is the trivial bundle. Let $L$ be a relative ample line bundle over $\mc{X}$, i.e. there exists a metric $\phi$ on $L$ such that its curvature $\omega:=\sqrt{-1}\p\b{\p}\phi$ is positive along each fiber. The $L^2$-metric (\ref{L2-metric}) on $\mc{A}^{0,0}$ is
\begin{align}\label{L2-metric0}
	(u,v)=\int_{X_t} u\b{v}\frac{\omega^n}{n!}.
\end{align}
We call a metric $\phi$ on $L$ is weak geodesic-Einstein with respect to $\omega_B$ if $tr_{\omega_B}c(\phi)=\pi^*f(z)$ for some function on $B$.
\begin{prop}\label{prop5.4}
	$\phi$ is a weak geodesic-Einstein metric on $L$ if and only if the metric $L^2$-metric (\ref{L2-metric0}) is a Hermitian-Einstein metric on $\mc{A}^{0,0}$.
\end{prop}
\begin{proof}
	By (\ref{Curvature of D}) and noting that $p=q=0$, $E$ is trivial, one has
	\begin{align}
	\begin{split}
		(D^{\mc{A}})^2u &=(\n^{\mc{A}})^2u=\sum [[d^E,\delta_{V_j}],[d^E,\delta_{\b{V}_k}]]dt^j\wedge d\b{t}^k\\
		&=\sum[V_j, V_{\b{k}}]u dt^j\wedge d\b{t}^k
		\end{split}
	\end{align}
for any $u\in \mc{A}^{0,0}$. By \cite[Lemma 6.1]{Wang1}, we have
$$[V_j, V_{\b{k}}]=(c(\phi)_{j\b{k}})_{\b{\lambda}}\phi^{\b{\lambda}\nu}\frac{\p}{\p v^{\nu}}-(c(\phi)_{j\b{k}})_{\nu}\phi^{\b{\lambda}\nu}\frac{\p}{\p\b{v}^{\lambda}},$$
where $c(\phi)_{j\b{k}}=\langle V_j,V_k\rangle_{\omega}$ is the coefficient of geodesic curvature $c(\phi)$.  Let $\omega_B=\sqrt{-1}g_{i\b{j}}dt^i\wedge d\b{t}^j$ be a Hermitian metric on $B$, then
\begin{align}
\begin{split}
\Lambda_{\omega_B}(D^{\mc{A}})^2u &=	(g^{j\b{k}}c(\phi)_{j\b{k}})_{\b{\lambda}}\phi^{\b{\lambda}\nu}\frac{\p}{\p v^{\nu}}u-(g^{i\b{k}}c(\phi)_{j\b{k}})_{\nu}\phi^{\b{\lambda}\nu}\frac{\p}{\p\b{v}^{\lambda}}u\\
&=(tr_{\omega_B}c(\phi))_{\b{\lambda}}\phi^{\b{\lambda}\nu}\frac{\p}{\p v^{\nu}}u-((tr_{\omega_B}c(\phi))_{\nu}\phi^{\b{\lambda}\nu}\frac{\p}{\p\b{v}^{\lambda}}u.
\end{split}
\end{align}
Thus, for any $u\in \mc{A}^{0,0}$, $\Lambda_{\omega_B}(D^{\mc{A}})^2u=\lambda u$ if and only if $\lambda=0$ and $tr_{\omega_B}c(\phi)=\pi^*f(t)$ for some function $f(z)$ on $B$, which completes the proof.
\end{proof}

Now we assume that $B$ is compact, then for any smooth function $f$,  there is a smooth solution $\tilde{f}$ solves
\begin{align}
\Delta \tilde{f}+f(t)=\int_B f(t) \frac{\omega^m_{B}}{m!}
\end{align}
where $\Delta:=g^{i\b{j}}\p_i\p_{\b{j}}$ and $m=\dim B$. Denote $\tilde{\phi}=\phi+\tilde{f}$. If $tr_{\omega_B}c(\phi)=\pi^*f(z)$, then
\begin{align}
tr_{\omega_B}c(\tilde{\phi})=\Delta \tilde{f}+f(z)=	\int_B f(t) \frac{\omega^m_{B}}{m!},
\end{align}
which implies that $\tilde{\phi}$ is a geodesic-Einstein metric on $L$. Combining with Proposition \ref{prop5.4}, we have
\begin{cor}\label{cor5.5}
If $B$ is compact, then up to a smooth function on 	$B$, $\phi$ is a  geodesic-Einstein metric on $L$ if and only if the metric $L^2$-metric (\ref{L2-metric0}) is a Hermitian-Einstein metric on $\mc{A}^{0,0}$.
\end{cor}

\end{document}